\documentclass{article}
\usepackage{amsfonts,amssymb,amsmath,amsthm}
\usepackage{url}

\def\thtext#1{
  \catcode`@=11
  \gdef\@thmcountersep{. #1}
  \catcode`@=12
}

\def\threst{
  \catcode`@=11
  \gdef\@thmcountersep{.}
  \catcode`@=12
}

 \catcode`@=11
 \def\.{.\spacefactor\@m}
 \catcode`@=12

\theoremstyle{plain}
\newtheorem{thm}{Theorem}[section]
\newtheorem{prop}{Proposition}[section]
\newtheorem{cor}[prop]{Corollary}
\newtheorem{lem}[prop]{Lemma}

\theoremstyle{definition}
\newtheorem{constr}[prop]{Construction}

\newcommand{\cH}{\mathcal{H}}
\newcommand{\cM}{\mathcal{M}}
\newcommand{\cP}{\mathcal{P}}
\newcommand{\cR}{\mathcal{R}}
\newcommand{\cT}{\mathcal{T}}

\newcommand{\R}{\mathbb{R}}

\newcommand{\s}{\sigma}

\newcommand{\hd}{\widehat{d}}

\newcommand{\diam}{\operatorname{diam}}
\newcommand{\dis}{\operatorname{dis}}
\newcommand{\opt}{{\operatorname{opt}}}
\newcommand{\smt}{\operatorname{smt}}
\newcommand{\SMT}{\operatorname{SMT}}

\renewcommand{\:}{\colon}
\newcommand{\0}{\emptyset}
\newcommand{\sm}{\setminus}
\renewcommand{\ss}{\subset}
\renewcommand{\sp}{\supset}
\newcommand{\x}{\times}

\begin{document}

\author{A.O.Ivanov, N.K.Nikolaeva, A.A.Tuzhilin}
\title{Steiner Problem in Gromov--Hausdorff Space: the Case of Finite Metric Spaces}
\date{}
\maketitle

\begin{abstract}
It is shown that each finite family of finite metric spaces, being considered as a subset of Gromov--Hausdorff space, can be connected by a Steiner minimal tree.
\end{abstract}

\section{Introduction}
\markright{\thesection.~Introduction}
The present paper is devoted to the Steiner problem in the space of compact metric spaces, endowed with Gromov--Hausdorff metric. It is shown that each boundary set consisting of finite metric spaces only can be connected by a Steiner minimal tree. In the general case, the authors have solved the Steiner problem for $2$-point boundaries~\cite{IvaNikolaevaTuz}, where the problem is equivalent to the fact that the ambient space is geodesic. General case of more than $2$ boundary points has resisted to the authors attempts based on the Gromov precompactness criterion. Nevertheless, we hope that the technique we worked out will be useful for either proving the theorem, or for constructing a counterexample.

\section{Main Definitions and Results}
\markright{\thesection.~Main Definitions and Results}
Let $X$ be an arbitrary metric space.  By $|xy|$ we denote the distance between points $x,\,y\in X$. Let $\cP(X)$ be the family of all nonempty subsets of $X$. For $A,\,B\in\cP(X)$ we put
$$
d_H(A,B)=\max\bigl\{\sup_{a\in A}\inf_{b\in B}|ab|,\,\sup_{b\in B}\inf_{a\in A}|ab|\bigr\}.
$$
The value $d_H(A,B)$ is called the \emph{Hausdorff distance between $A$ and $B$}.

Notice that $d_H(A,B)$ as may be equal to infinity (e.g., for $X=A=\R$ and $B=\{0\}\ss\R$), so as may vanish for non-equal $A$ and $B$ (e.g., for $X=\R$, $A=[a,b]$, and $B=[a,b)$).

Let $\cH(X)\ss\cP(X)$ denote the set of all nonempty closed bounded subsets of $X$.

\begin{prop}[\cite{BurBurIva}]
The restriction of $d_H$ onto $\cH(X)$ is a metric.
\end{prop}

Let $X$ and $Y$ be metric spaces. The triple $(X',Y',Z)$ consisting of a metric space $Z$ and its subsets $X'$ and $Y'$ which are isometric to $X$ and $Y$, respectively, is called a \emph{realization of the pair $(X,Y)$}. We put
$$
d_{GH}(X,Y)=\inf\bigl\{r:\exists (X',Y',Z),\,d_H(X',Y')\le r\bigr\}.
$$
The value $d_{GH}(X,Y)$ is called the \emph{Gromov--Hausdorff distance between $X$ and $Y$}.

By $\cM$ we denote  the set of all compact metric spaces considered up to an isometry.

\begin{prop}[\cite{BurBurIva}]
The restriction of $d_{GH}$ onto $\cM$ is a metric.
\end{prop}

The Gromov--Hausdorff distance can be effectively investigated in terms of correspondences.

Let $X$ and $Y$ be arbitrary nonempty sets. We put $\cP(X,Y)=\cP(X\x Y)$. The elements of $\cP(X,Y)$ are called \emph{relations\/} between $X$ and $Y$. If $X'\ss X$ and $Y'\ss Y$ are nonempty subsets, and $\s\in\cP(X,Y)$, then we put
$$
\s|_{X'\x Y'}=\bigl\{(x,y)\in\s:x\in X',\,y\in Y'\bigr\}.
$$
Notice that $\s|_{X'\x Y'}$ may be empty and, thus, may not belong to $\cP(X',Y')$.

Let $\pi_X\:(x,y)\mapsto x$ and $\pi_Y\:(x,y)\mapsto y$ be the canonical projections. A relation $\s\in\cP(X,Y)$ is called a \emph{correspondence}, if the restrictions of $\pi_X$ and $\pi_Y$ onto $\s$ are surjective. By $\cR(X,Y)$ we denote the set of all correspondences between $X$ and $Y$.

If $X$ and $Y$ are metric spaces, then for each relation $\s\in\cP(X,Y)$ its \emph{destortion\/} is defined as
$$
\dis\s=\sup\Bigl\{\bigl||xx'|-|yy'|\bigr|: (x,y),\,(x',y')\in\s\Bigr\}.
$$

\begin{prop}[\cite{BurBurIva}]\label{prop:GH-metri-and-relations}
Let $X$ and $Y$ be metric spaces. Then
$$
d_{GH}(X,Y)=\frac12\inf\bigl\{\dis R:R\in\cR(X,Y)\bigr\}.
$$
\end{prop}

For a metric space $X$m be $\diam X$ we denote its \emph{diameter\/}: $\diam X=\sup\bigl\{|xy|:x,\,y\in X\bigr\}$.

\begin{cor}[\cite{BurBurIva}]\label{cor:below-bound}
For any metric spaces $X$ and $Y$ such that the diameter of at least one of them is finite, we have
$$
d_{GH}(X,Y)\ge\frac12|\diam X-\diam Y|.
$$
\end{cor}

A correspondence $R\in\cR(X,Y)$ is called \emph{optimal}, if $d_{GH}(X,Y)=\frac12\dis R$. By $\cR_\opt(X,Y)$  we denote the set of all optimal correspondences between $X$ and $Y$.

\begin{prop}[\cite{IvaIliTuz}, \cite{Memoli}, \cite{blog}]\label{prop:optimal-cor}
For $X,\,Y\in\cM$ we have $\cR_\opt(X,Y)\ne\0$.
\end{prop}

Let $\cM_n\ss\cM$ consist of all metric spaces containing at most $n$ points; and let $\cM(d)\ss\cM$ consist of all spaces, whose diameters are at most $d$; at last, put $\cM_n(d)=\cM_n\cap\cM(d)$.

\begin{prop}[\cite{BurBurIva}]\label{prop:kompakt}
The space $\cM_n(d)$ is compact.
\end{prop}

Recall that a \emph{simple graph\/} is a pair $(V,E)$ consisting of a finite set $V$ and some set $E$ of $2$-element subsets of $V$. For reasons of convenience, we write $vw$ instead of $\{v,w\}$. For general Graph Theory terminology one can consult~\cite{Harary}. Since we consider only simple graphs,  we shall omit the word ``simple''.

Let $X$ be an arbitrary set and $G=(V,E)$ be a graph such that $V\ss X$. In this case we say that $G$ is a graph \emph{on the set $X$}. Let $M\ss X$ be an arbitrary finite subset and $G=(V,E)$ be a \textbf{connected\/} graph on $X$, such that $V\sp M$. For such a graph we say that it \emph{connects $M$}; the vertices from $M$ we call  \emph{boundary\/} ones, and the vertices from $V\sm M$ are referred as \emph{interior\/} ones.

Let $G=(V,E)$ be a graph on a metric space $X$. For an edge $e=vw\in E$, the \emph{length $|e|$} of $e=vw$ is the distance $|vw|$ between its vertices. The \emph{length $|G|$ of the graph $G$} is the sum of the lengths of all the edges of the graph.

For each finite subset $M$ of a metric space $X$ the value
$$
\smt(M,X)=\inf\bigl\{|G|:\text{$G$ is a graph on $X$ connecting $M$}\bigr\}
$$
is called the \emph{length of Steiner minimal tree on  $M$}.

The next result is obvious.

\begin{prop}\label{prop:minimizing-trees}
For each finite subset $M$ of a metric space $X$ we have
$$
\smt(M,X)=\inf\bigl\{|G|:\text{$G$ is a tree on $X$ connecting $M$}\bigr\}.
$$
\end{prop}

By $\SMT(M,X)$ we denote the set of all graphs $G$ on $X$ connecting $M\ss X$ and such that $|G|=\smt(M,X)$. Notice that $\SMT(M,X)$ may be empty. In the case when $\SMT(M,X)\ne\0$, each graph $G\in\SMT(M,X)$ does not contain cycles and is called a \emph{Steiner minimal tree on $M$}.

The technique developed in~\cite{IvaTuz} for Riemannian manifolds can be obviously generalized to proper metric spaces.

\begin{prop}\label{prop:SMT-existence}
Let $X$ be a proper metric space. Then for each nonempty finite $M\ss X$ we have $\SMT(M,X)\ne\0$.
\end{prop}

The next result follows from ~\ref{prop:kompakt} and~\ref{prop:SMT-existence}.

\begin{cor}\label{cor:bound_SMT_existence}
For any nonempty finite set $M\ss\cM_n(d)$ we have
$$
\SMT\bigl(M,\cM_n(d)\bigr)\ne\0.
$$
\end{cor}

The next Theorem is the main result of the present paper.

\begin{thm}
For each $M=\{m_1,\ldots,m_k\}\ss\cM_n$ we have
$$
\SMT(M,\cM)\ne\0.
$$
\end{thm}

\begin{proof}
Put $r=\smt(M,\cM)$, and let $\cT(M)$ consist of all the trees $G$ on $\cM$ connecting $M$ and such that $|G|\le r+1$. By definition of $\smt$ and by~\ref{prop:minimizing-trees}, we have $\cT(M)\ne\0$ and
$$
\smt(M,\cM)=\inf\bigl\{|G|:G\in\cT(M)\bigr\}.
$$

Choose an arbitrary graph $G=(V,E)\in\cT(M)$.

\begin{lem}\label{lem:diam int point}
Put $d=\max_i\{\diam m_i\}$ and $\hd=2r+d+2$, then $V\ss\cM(\hd)$.
\end{lem}

\begin{proof}
If there exists a vertex $v\in V$, whose diameter is greater than $\hd$, then, by~\ref{cor:below-bound}, we have
$$
|G|\ge d_{GH}(v,m_1)\ge\frac12|\diam v-\diam m_1|>\frac12(2r+d+2-d)=r+1,
$$
a contradiction.
\end{proof}

\begin{constr}
For each $e=vw\in E$ choose an $R_e\in\cR_\opt(v,w)$ (it does exist by~\ref{prop:optimal-cor}). It is easy to see that for each $x\in \sqcup_im_i$ there exists a tree $T_x=(P_x,F_x)$ such that $P_x$ is obtained from $V$ by choosing just one point $p_v$ in each compact space $v\in V$ provided $p_vp_w\in R_e$ for each $vw\in E$; the edge set $F_x$ of the tree $T_x$ consists exactly of all such pairs $p_vp_w$. Thus, the mapping $v\mapsto p_v$ is an isomorphism of the graphs $G$ and $T_x$. We call the tree $T_x$ by a \emph{thread emitted from $x$}.
\end{constr}

Let $m=\sqcup_{i=1}^km_k$ and $N$ be the number of points in $m$. For each $x\in m$ consider a thread $T_x$ emitted from $x$. For each $v\in V$ we define $v'\ss v$ as
$$
v'=\{y\in v:\exists x\in m,\,y\in P_x\}.
$$
If $e=vw$, then put $e'=v'w'$. Notice that for each $v\in V$ we have $v'\in\cM_N$.

Let $V'=\{v'\}_{v\in V}$. By $G'=(V',E')$ we denote the graph such that $v'w'\in E'$, iff $vw\in E$. The above results imply that $G'$ is a graph on $\cM_N$. Clearly that the mapping $v\mapsto v'$ is an isomorphism from $G$ to $G'$.

The next property of the $v'\ss v$ follows immediately from their construction.

\begin{lem}\label{lem:low-length}
For any $e=vw\in E$ we have
$$
R_{e'}:=R_e|_{v'\x w'}\in\cR(v',w')\ \text{and}\ \dis R_{e'}\le\dis R_e.
$$
In particular, $|e'|\le|e|$, thus, $|G'|\le|G|$ and, therefore, $G'\in\cT(M)$.
\end{lem}

Now apply~\ref{lem:diam int point} and~\ref{lem:low-length}.

\begin{cor}
The $G'\in\cT(M)$ constructed above is a tree on $\cM_N(\hd)$, therefore,
$$
\smt(M,\cM)=\smt(M,\cM_N(\hd)).
$$
\end{cor}

It remains to apply~\ref{cor:bound_SMT_existence}.
\end{proof}

\end{document}